\documentclass{amsart}

\usepackage[utf8]{inputenc}
\usepackage[T1]{fontenc}
\usepackage{textcomp}
\usepackage{amsfonts}
\usepackage{amsthm}
\usepackage{amssymb}
\usepackage{pst-node}
\usepackage[all]{xy}
\usepackage{mathtools}
\usepackage{chngcntr}
\usepackage{thmtools}
\usepackage{mathrsfs}
\usepackage{tcolorbox}
\usepackage{listings}
\usepackage{hyperref}
\usepackage{color}
\usepackage{float}
\usepackage{tikz, tikz-cd}
\usepackage{enumerate}
\usepackage[nameinlink, capitalize, noabbrev]{cleveref}
\usepackage{natbib}
\usepackage{caption}
\usetikzlibrary{intersections}
\usepackage{pgfplots}
\pgfplotsset{width=10cm,compat=1.9}
\usepgfplotslibrary{external}
\tikzset{mytext/.style={font=\small, text=black}}
\tikzset{main node/.style={circle,fill=lime!30,draw,minimum size=0.5cm,inner sep=0pt},}

\hypersetup{ colorlinks=true,linkcolor=teal,    citecolor=magenta, }

\def\set#1{{\def\st{\;:\;}\left\{#1\right\}}}
\def\abs#1{\left\vert{#1}\right\vert}
\def \<#1>{{\left\langle{#1}\right\rangle}}

\def\ZZ{\mathbb Z}
\def\QQ{\mathbb Q}
\def\PP{\mathbb P}
\def\OK{\mathcal{O}_K}
\def\FF{\mathbb F}
\def\CC{\mathbb C}

\def\Q-{\overline{\mathbb Q}}

\DeclareMathOperator{\Gal}{Gal}

\DeclareMathOperator{\FPP}{FPP}
\DeclareMathOperator{\Per}{Per}
\DeclareMathOperator{\car}{char}

\DeclareMathOperator{\st}{st}
\DeclareMathOperator{\tr}{tr}

\DeclareMathOperator{\rad}{rad}

\DeclareMathOperator{\lcm}{lcm}
\DeclareMathOperator{\Frob}{Frob}

\newtheorem{Theorem}{Theorem}

\newtheorem{Corollary}[Theorem]{Corollary}
\newtheorem{proposition}{Proposition}[section]
\newtheorem{lemma}[proposition]{Lemma}

\newtheorem{theorem}[proposition]{Theorem}

\newtheorem{Question}[proposition]{Question}
\theoremstyle{definition}
\newtheorem{remark}[proposition]{Remark}
\newtheorem{Definition}[proposition]{Definition}

\title{Proportion of periodic points in reduction of polynomials}
\author{Santiago Radi}
\date{March 2026}
\address{Santiago Radi: Department of Mathematics, Texas A\&M University, 77843 College Station, U.S.A.
}
\email{santiradi@tamu.edu}
\thanks{The author is supported by Grigorchuk's Simons Foundation Grant MP-TSM-00002045 and the department of Mathematics of Texas A\&M University.}

\begin{document}

\begin{abstract}
In 2014, Juul, Kurlberg, Madhu and Tucker asked the following: given $K$ a number field and $f$ a rational function with coefficients in $K$, if $f_\mathfrak{p}$ denotes the reduction of $f$ modulo a prime ideal $\mathfrak{p}$ in the ring of integers of $K$, what is the limit inferior of the proportion of periodic points of $f_\mathfrak{p}$ when the norm of $\mathfrak{p}$ goes to infinity? 

Recent results of Fariña-Asategui and the author show that when $f$ is a polynomial of degree $d \geq 2$ non-linearly conjugate over $\CC$ to a Chebyshev polynomial then the limit inferior is zero.  

In this article, we address the remaining cases to give a complete classification of the problem in the case of polynomials.
\end{abstract}

\maketitle

\section{Introduction}

Let $K$ be a number field, $\OK$ its ring of integers and $f$ a polynomial with coefficients in $K$. For all but finitely many prime ideals $\mathfrak{p}$ in $\OK$, we can consider the reduction map $f_\mathfrak{p}$ acting on the residue field $\FF_\mathfrak{p} := \OK/\mathfrak{p}$. 

As $\FF_\mathfrak{p}$ is a finite field, a point $x \in \FF_\mathfrak{p}$ is either periodic or strictly preperiodic for $f_\mathfrak{p}$. How does the proportion of periodic points vary as the norm of the prime $\mathfrak{p}$ goes to infinity? 

The first problem in this direction was formulated by Schur in 1923, who asked for a classification of all polynomials with coefficients in $\QQ$ whose reduction is a bijection for infinitely many primes \cite{Schur1923} (note that if the reduction is a bijection, every point must be periodic and consequently the proportion of periodic points is $1$). The question was answered in 1970 by Fried over any number field $K$, showing that a polynomial whose reduction is a bijection for infinitely many primes must be a composition of linear polynomials with coefficients in $K$ and Dickson polynomials, where Dickson polynomials are defined by $f(x+a/x) = x^n + a^n/x^n$ with $a \in K$; see \cite{Fried1970}. In 2002, Guralnick, Müller and Saxl generalized Fried's result to rational functions over any number field, using the theory of arithmetic and geometric iterated Galois groups \cite{GuralnickMullerSaxl2003}. These maps are referred to in \cite{GuralnickMullerSaxl2003} as arithmetically exceptional.

Although we find infinitely many primes where the proportion of periodic points is $1$ for arithmetically exceptional maps, the proportion of periodic points for the other primes might be arbitrarily small. In fact, this is what we expect to happen in a random map. In \cite{FlajoletOdlyzko1990}, Flajolet and Odlyzko proved that if $f: S \rightarrow S$ is a random map, where $S$ is finite, the number of periodic points is asymptotically $$\sqrt{\frac{\pi \cdot \# S}{2}}$$ when $\# S \rightarrow +\infty$. Therefore, based on this heuristic, we would expect to find prime ideals $\mathfrak{p}$ where the reduction $f_\mathfrak{p}$ has a proportion of periodic points arbitrarily small. Unfortunately, the assumption that the reductions of a polynomial behave randomly is far from being true. This leads to the following question:

\begin{Question}[{\cite[Question 1.4]{Juul2014}}]
\label{question: proportion of periodic in reduction of maps}
Given $K$ a number field and $f \in K[x]$ a polynomial with coefficients in $K$. Denote $\Per(f,\FF_\mathfrak{p}) := \set{\alpha \in \FF_\mathfrak{p}: \alpha \text{ is periodic for } f_\mathfrak{p}}$. What is the value of 
\begin{align*}
\Per_{\inf}(f,K) := \liminf_{N(\mathfrak{p}) \rightarrow +\infty} \frac{\# \Per(f,\FF_\mathfrak{p})}{N(\mathfrak{p})}?
\end{align*}
\end{Question}

\cref{question: proportion of periodic in reduction of maps} was firstly formulated in 2014 by Juul, Kurlberg, Madhu and Tucker \cite{Juul2014}, who asked for a classification of the values of $\Per_{\inf}(f,K)$ for every number field $K$ and any rational function $f \in K(x)$; \cite[Question 1.4]{Juul2014}. 

In \cite{Juul2014}, the authors proved that $$\Per_{\inf}(x^d+c,K) = 0$$ for any $d \geq 2$ and $c \in K$ such that $0$ is not preperiodic, allowing them to give a complete classification of the problem in the case of quadratic polynomials over $\QQ$; \cite[Theorem 1.5]{Juul2014}. In \cite[Theorem 1.3]{Juul2014}, they also proved that if $f$ has two critical points whose forward orbits do not intersect, then $\Per_{\inf}(f,K) = 0$. 

It is expected that in general $\Per_{\inf}(f,K) = 0$ and this has been confirmed in most of the examples known. In accordance with this, in \cite[Theorem 1.2]{Juul2014} authors showed that given $K$ a number field and numbers $\epsilon > 0$ and $d \geq 2$, there is an open Zariski dense set $U_{d,\epsilon}$ in the moduli space of rational functions such that $$\Per_{\inf}(f,K) \leq \epsilon$$ for all $f \in U_{d,\epsilon}$. 

The only cases known where $\Per_{\inf}(f,K) > 0$ correspond to the Chebyshev polynomials. Recall that the Chebyshev polynomial of degree $d$, denoted $T_d$, is defined as the unique polynomial with coefficients in $\ZZ$ satisfying
\begin{align}
T_d \left( x + \frac{1}{x} \right) = x^d + \frac{1}{x^d}.
\label{equation: definition Chebyshev polynomial}
\end{align}

In \cite[Example 7.2]{Juul2014} is proved that if $d \geq 2$ then
\begin{align}
\Per_{\inf}(T_d,\QQ) = \left \{ \begin{matrix} 
1/2 & \mbox{if $d$ is an odd prime power,} \\ 
1/4 & \mbox{if $d$ is an even prime power,} \\
0 & \mbox{otherwise.}
\end{matrix}\right.
\label{equation: Per in chebyshev over Q}
\end{align} 

In this article, we give a complete answer to \cref{question: proportion of periodic in reduction of maps}. The cases where $f$ is a polynomial of degree $d = 0$ or $1$ are straightforward and will be solved in \cref{lemma: case degree 1}. 
When $d \geq 2$, we rely on the tools of arboreal representations, particularly, the results related to iterated Galois groups and fixed-point proportion. In 1985, Odoni initiated the study of iterated Galois groups and fixed-point proportion to solve problems in number theory \cite{Odoni1980}. In his pioneering article, he developed these tools to prove that the Dirichlet density of prime numbers dividing at least one term in the Sylvester sequence (defined recursively as $w_0 = 2$ and $w_{n+1} = 1 + w_1 \dots w_n$) is zero. His idea was then generalized to solve other problems in number theory. Specifically, in 2014 Juul, Kurlberg, Madhu and Tucker proved that 
\begin{align}
\Per_{\inf}(f,K) \leq \FPP(G_\infty(\CC,f,t)),
\label{equation: upper bound PerInf FPP intro}
\end{align}
where $\FPP$ is the fixed-point proportion, $G_\infty(\CC,f,t)$ is the iterated Galois group of $f$ over $\CC$ and $t$ is transcendental over $\CC$ (see \cref{section: Polynomials not linearly conjugate over C to T_d} for the unexplained terms here and \cite{Jones2014} for a survey about the topic). 

The recent breakthrough of Fariña-Asategui and the author  completely classifies the fixed-point proportion for geometric iterated Galois groups of polynomials \cite[Theorem 2]{FariñaRadi2026FPP}. As $\CC$ is algebraically closed, the result applies to this case obtaining that 
$\FPP(G_\infty(\CC,f,t)) = 0$ if $f$ is not linearly conjugate over $\CC$ to $\pm T_d$. Therefore $$\Per_{\inf}(f,K) = 0$$ if $f$ is not linearly conjugate over $\CC$ to $\pm T_d$. 

In the case where $f$ is linearly conjugate over $\CC$ to $\pm T_d$, then the fixed-point proportion of $G_\infty(\CC,f,t)$ is positive and therefore the upper bound in \cref{equation: upper bound PerInf FPP intro} is not enough to calculate $\Per_{\inf}(f,K)$. For these cases, we will work out $\Per_{\inf}(f,K)$ directly. We will prove that if $f$ is linearly conjugate over $\CC$ to $\pm T_d$ then
\begin{align}
\Per_{\inf}(f,K) = \Per_{\inf}(T_d,K).
\label{equation: Perinf f = Perinf Td intro}
\end{align}

Although this step can look obvious, it is not obvious if the conjugation does not have coefficients in $K$. We shall see in \cref{section: Polynomials linearly conjugate over C to T_d} that when the conjugation is not in $K$, it lies in a field extension $K'/K$ of degree $2$. If $\mathfrak{p}$ is a prime ideal in $\OK$ that is inert in $\mathcal{O}_{K'}$, then the periodic points of $f$ in $\FF_\mathfrak{p}$ are in bijection with the periodic points of $T_d$. But if $\mathfrak{p}$ is split in $\mathcal{O}_{K'}$, then the periodic points of $f$ are in bijection with the periodic points of $T_d$ in a 2-degree extension of $\FF_\mathfrak{p}$ and some work needs to be done to prove that $\Per_{\inf}(f,K)$ does not change when we restrict to the periodic points of $f$ in $\FF_\mathfrak{p}$ (see \cref{proposition: Perinf f with a not in K} for a proof of this fact).

Once we know \cref{equation: Perinf f = Perinf Td intro}, we reduce the problem to calculate $\Per_{\inf}(T_d,K)$. The calculation is based on the ideas used by the authors in \cite[Example 7.2]{Juul2014} for the case of $K = \QQ$, obtaining that
\begin{align*}
\Per_{\inf}(T_d,K) = \frac{1}{2} \liminf_{N(\mathfrak{p}) \rightarrow +\infty} \frac{r(N(\mathfrak{p})-1,d) + r(N(\mathfrak{p})+1,d)}{N(\mathfrak{p})},
\end{align*}
where $r(a,b)$ is defined as the largest divisor of $a$ coprime to $b$.  

So, in order to calculate $\Per_{\inf}(T_d,K)$ we need to find infinitely many prime ideals in $\OK$ such that $N(\mathfrak{p})-1$ and $N(\mathfrak{p})+1$ have large divisors that are not coprime to $d$, since this will make 
$r(N(\mathfrak{p})-1,d)$ and $r(N(\mathfrak{p})+1,d)$ relatively small compared to $N(\mathfrak{p})$. The main tool to find such prime ideals will be Cebotarev's theorem. However, the existence of roots of unity in $K$ will impose restrictions of the prime factors dividing $N(\mathfrak{p})-1$ for all but finitely many prime ideals. Concretely, if $m_K$ is the maximal positive number such that $K$ contains a $m_K$th root of unity, then for all but finitely many prime ideals we have $m_K$ divides $N(\mathfrak{p})-1$. The relation between the prime factors of $m_K$ and $d$ will allow to reduce more or less $r(N(\mathfrak{p})-1,d)$ and $r(N(\mathfrak{p})+1,d)$, leading to cases where $\Per_{\inf}(T_d,K)$ is zero and others where $\Per_{\inf}(T_d,K)$ is positive. 

In order to state the main theorem, recall that the radical of a number $d$, denoted $\rad(d)$, is the product of the prime divisors of $d$. Therefore, note that when we say that $\rad(d) \mid \rad(m)$, this means that all prime divisors of $d$ also divide $m$. Contrarily, when we say $\rad(d) \nmid \rad(m)$, this means that there exists one prime $p$ such that $p \mid d$ and $p \nmid m$. Given a prime $p$, the valuation of $m$ at $p$ will be denoted $v_p(m)$.

\begin{Theorem}
Let $K$ be a number field and let $m_K$ denote the greatest positive number such that $K$ contains a $m_K$th primitive root of unity. Let $f$ be a polynomial with coefficients in $K$ and degree $d$.

\begin{enumerate}
\item If $d = 0$ then $$\Per_{\inf}(f,K) = 0.$$

\item If $d = 1$ then $$\Per_{\inf}(f,K) = 1.$$

\item If $d \geq 2$ and $f$ is not linearly conjugate over $\CC$ to $\pm T_d$, then $$\Per_{\inf}(f,K) = 0.$$

\item If $d \geq 2$ and $f$ is linearly conjugate over $\CC$ to $\pm T_d$, then $$\Per_{\inf}(f,K) = \Per_{\inf}(T_d,K).$$

\item If $d$ has at least two prime factors and $\rad(d) \nmid \rad(m_K)$, then $$\Per_{\inf}(T_d,K) = 0.$$

\item If $d$ is odd, has at least two prime factors and $\rad(d) \mid \rad(m_K)$, then $$\Per_{\inf}(T_d,K) = \frac{1}{2}.$$  

\item If $d$ is even, has at least two prime factors, $\rad(d) \mid \rad(m_K)$ and $v_2(m_K) \geq 2$, then $$\Per_{\inf}(T_d,K) = \frac{1}{4}.$$

\item If $d$ is even, has at least two prime factors, $\rad(d) \mid \rad(m_K)$ and $v_2(m_K) = 1$, then $$\Per_{\inf}(T_d,K) = 0.$$

\item If $d$ is an odd prime power, then $$\Per_{\inf}(T_d,K) = \frac{1}{2}.$$ 

\item If $d$ is an even prime power, then $$\Per_{\inf}(T_d,K) = \frac{1}{4}.$$ 
\end{enumerate}
\label{Theorem: main result clasification}
\end{Theorem}

In the case that $K = \QQ$ then $m_K = 2$ and only the cases (1), (2), (3), (4), (5), (9) and (10) apply. Moreover, the results obtained in \cref{Theorem: main result clasification} for the case $K = \QQ$ are consistent with \cref{equation: Per in chebyshev over Q}. 

\begin{Corollary}
If $f$ is a polynomial with coefficients in $\QQ$ and degree $d \geq 1$. Then 
\begin{align*}
\Per_{\inf}(f,\QQ) = \left \{ \begin{matrix} 
1 & \mbox{if $d = 1$,} \\ 
1/2 & \mbox{if $d$ odd prime power and $f$ conjugate  over $\CC$ to $\pm T_d$,} \\ 
1/4 & \mbox{if $d$ even prime power and $f$ conjugate over $\CC$ to $\pm T_d$ ,} \\
0 & \mbox{otherwise.}
\end{matrix}\right.
\label{equation: Per in chebyshev over Q}
\end{align*} 
\label{Corollary: main result case Q}
\end{Corollary}

Although \cref{equation: upper bound PerInf FPP intro} also holds in the case of rational functions, there is not a complete classification of the fixed-point proportion of $G_\infty(\CC,f,t)$ in the case of rational maps. In \cite{FariñaRadi2026FPP}, the authors proved that $\FPP(G_\infty(\CC,f,t)) = 0$ for rational maps whose critically exceptional set does not have more than one element (see \cite[Theorem 6]{FariñaRadi2026FPP}) and conjectured that the cases of positive fixed-point proportion should come from rational maps with euclidean orbifold of type different to $(\infty, \infty)$ (see \cite[Conjecture 7]{FariñaRadi2026FPP}). This latter family includes Lattès maps $L$, although the value of $\FPP(G_\infty(\CC,L,t))$ is not known for any of them. Regarding $\Per_{\inf}(L,K)$, authors in \cite[Example 7.3]{Juul2014} proved that in many cases, Lattès maps $L$ induced by multiplication by a prime $\ell$ in a elliptic curve defined over $\QQ$ must have $\Per_{\inf}(L,K) = 0$. Unlike Chebyshev polynomials, there are not known cases of Lattès maps such that $\Per_{\inf}(L,K) > 0$. Therefore, it still remains open a whole classification in the case of rational functions:

\begin{Question}
What are the values of $\Per_{\inf}(f,K)$ for every number field $K$ and every rational function $f \in K(x)$?
\end{Question}

\section{Preliminaries}

\subsection{Iteration of maps and local degree}

Given a set $S$ and a function $f: S \rightarrow S$, we use the notation $f^n$ to refer to the $n$th composition of $f$ with itself. By convention, $f^0$ equals the identity map in $S$.

Given $\alpha \in S$, we say that $\alpha$ is \textit{preperiodic} if there exist $m > n \geq 0$ such that $f^n(\alpha) = f^m(\alpha)$. By the pigeonhole principle, if $S$ is finite every point is preperiodic. We say that $\alpha$ is \textit{periodic} if $n$ can be taken zero in the definition of preperiodic. We denote 
$$\Per(f,S) := \set{\alpha \in S: \text{ $\alpha$ is periodic for $f$}}.$$ We say that a point $\alpha$ is \textit{strictly preperiodic} if it is PREperiodic but not periodic.

\begin{Definition}
Let $F$ be a field, $\overline{F}$ an algebraic closure of $F$ and $f,g \in \overline{F}(x)$. We say that $f$ and $g$ are \textit{linearly conjugate} over $F$ if there exists $L \in F(x)$ of degree $1$ such that $g = L^{-1} \circ f \circ L$.  
\end{Definition}

Given $\alpha \in \PP^1(F)$ and $f \in F(x)$, the \textit{local degree} of $\alpha$ at $f$ is $$e_f(\alpha) := \min \set{k \geq 1: f^{(k)}(\alpha) \neq 0}$$ where $f^{(k)}$ corresponds to the $k$th derivative of $f$. We say that $\alpha$ is a \textit{critical point} for $f$ if $e_f(\alpha) > 1$. We say that $\alpha$ is \textit{totally ramified} for $f$ if $e_f(\alpha) = \deg(f)$. If $f$ is a polynomial, then $\infty$ is a totally ramified point. 

\begin{lemma}
If $F$ is a field, we have polynomials $f,g \in F[x]$ linearly conjugate over $\overline{F}$ by a function $L$ and $f$ has only $\infty$ as a totally ramified point, then $L$ is a polynomial. 
\label{lemma: L is a polynomial}
\end{lemma}

\begin{proof}
Write $g = L^{-1} \circ f \circ L$ and suppose that $L$ is not a polynomial. Then $L(\infty) \neq \infty$. By multiplicativity of the local degree and the fact that $e_L(\alpha) = 1$ for every $\alpha \in \PP^1(F)$, then $$\deg(f) = e_g(\infty) = e_{L^{-1}}(f(L(\infty))) \cdot e_f(L(\infty)) \cdot e_L(\infty) = e_f(L(\infty)),$$ which implies that $L(\infty)$ is totally ramified for $f$ giving a contradiction.
\end{proof}

\subsection{Algebraic number theory}

In this subsection we summarize some results from algebraic number theory that we will use in the article. Choose $\iota$ an embedding of $\Q-$ in $\CC$ and given $n \geq 1$, define $\xi_n$ in $\Q-$ by $$\iota(\xi_n) = e^{\frac{2 \pi i}{n}}.$$

By a slight abuse of notation, we will write $\xi_n = e^{\frac{2 \pi i}{n}}$ with no emphasis in the embedding $\iota$.

\begin{lemma}
Let $K$ be a number field, $n,m \geq 1$ and suppose that $K$ contains primitive roots of unity of order $m$ and $n$. Then $K$ contains $\xi_{\lcm(m,n)}$.
\label{lemma: composition of roots of unity}
\end{lemma}

\begin{proof}
Taking powers if necessary, we may assume that $\xi_m$ and $\xi_n$ are in $K$.

By Bezout's theorem there exist $a,b \in \ZZ$ such that $am+bn = \gcd(m,n)$. Then
\begin{equation*}
\xi_m^b \xi_n^a = e^{2 \pi i \left(\frac{b}{m} + \frac{a}{n} \right)} = e^{2 \pi i \left(\frac{am+bn}{mn} \right)} = e^{2 \pi i \left(\frac{\gcd(m,n)}{mn} \right)} = e^{2 \pi i \left(\frac{1}{\lcm(m,n)} \right)}. \qedhere
\end{equation*}
\end{proof}

As $K$ is a number field, then $K$ is a finite extension of $\QQ$ and therefore \cref{lemma: composition of roots of unity} justifies the existence of a maximal number $m_K \geq 1$ such that $K$ contains $\xi_{m_K}$.

\begin{lemma}
\label{lemma: Galois group cyclotomic extension}
Let $n \geq 1$ and $K$ a number field. Then 
\begin{align*}
\Gal \left( K(\xi_n)/K \right) \cong \set{r \in (\ZZ/\lcm(n,m_K)\ZZ)^\times: r \equiv 1 \pmod{m_K}}.
\end{align*}
\end{lemma}

\begin{proof}
By \cref{lemma: composition of roots of unity} $K(\xi_n) = K(\xi_{\lcm(m_K,n)})$. If $\sigma \in \Gal(K(\xi_n)/K)$ then $$\sigma(\xi_{\lcm(m_K,n)}) = (\xi_{\lcm(m_K,n)})^{\varphi(\sigma)}$$ for some $\varphi(\sigma) \in (\ZZ/\lcm(n,m_K)\ZZ)^\times$. Moreover, if $\lcm(m_K,n)/c = m_K$ then 
\begin{align*}
\sigma(\xi_{m_K}) = \sigma((\xi_{\lcm(m_K,n)})^c) = (\xi_{\lcm(m_K,n)})^{\varphi(\sigma)c} = (\xi_{m_K})^{\varphi(\sigma)}.
\end{align*}
But $\sigma$ must fix $\xi_{m_K}$ so $\varphi(\sigma) \equiv 1 \pmod{m_k}$. By maximality of $m_K$, we obtain the bijection.
\end{proof}

The following result will be crucial in \cref{section: analysis of Perinf(Td)} and is based on Cebotarev's theorem:

\begin{proposition}
Let $K$ be a number field. Let $\set{n_i}_{i = 1}^k$ be a set of positive numbers and $\set{a_i}_{i = 1}^k$ another set of numbers such that $a_i \in (\ZZ/n_i\ZZ)^\times$. If the system of congruences 
\begin{align*}
\left \{ \begin{matrix} 
x &\equiv &1 \pmod{m_K} \\
x &\equiv &a_1 \pmod{n_1} \\
& \vdots &\\
x &\equiv &a_k \pmod{n_k} \\
\end{matrix}\right.
\end{align*} 
has a solution in $\ZZ$, then for infinitely many prime ideals $\mathfrak{p}$ in $\OK$ we have $$n_i \mid (N(\mathfrak{p})-a_i)$$ for all $i = 1, \dots, k$.
\label{proposition: cebotarev in my case}
\end{proposition}

\begin{proof}
Let $r$ be a solution and write $n = \lcm(n_1, \dots, n_k)$. By the Chinese remainder theorem we have $r \in (\ZZ/\lcm(m_K,n)\ZZ)^\times$. By \cref{lemma: Galois group cyclotomic extension}, we have that $r$ corresponds to an element $\sigma \in \Gal(K(\xi_n)/K)$. Then, by Cebotarev's theorem, there exist infinitely many prime ideals $\mathfrak{p}$ in $\OK$ such that 
$$\Frob_\mathfrak{p}(\xi_{\lcm(m_K,n)}) = (\xi_{\lcm(m_K,n)})^r.$$ 

Write $n_i = \lcm(m_K,n)/c_i$. Then, 
\begin{align*}
\Frob_\mathfrak{p}(\xi_{n_i}) = \Frob_\mathfrak{p}((\xi_{\lcm(m_K,n)})^{c_i}) = (\xi_{\lcm(m_K,n)})^{r \cdot c_i} = (\xi_{n_i})^r = (\xi_{n_i})^{a_i}.
\end{align*}
On the other hand, $$\Frob_\mathfrak{p}(\xi_{n_i}) = (\xi_{n_i})^{N(\mathfrak{p})},$$ so combining both equations we obtain that $$N(\mathfrak{p}) \equiv a_i \pmod{n_i}$$ which implies that 
\begin{equation*}
n_i \mid (N(\mathfrak{p})-a_i). \qedhere
\end{equation*}
\end{proof}

\subsection{Chebyshev polynomials}

Given $d \geq 0$, a field $F$ and $\zeta \in \overline{F}$ we define the \textit{twisted Chebyshev polynomial} of degree $d$ with parameter $\zeta$ as 
\begin{align*}
T_{d,\zeta} \left( x + \frac{\zeta}{x} \right) := x^d + \frac{\zeta}{x^d}.
\end{align*}
Note that when $\zeta = 1$, this corresponds to the standard Chebyshev polynomial defined in \cref{equation: definition Chebyshev polynomial}.

When $d \geq 1$, then $\zeta$ must be a $(d-1)$th root of unity. Indeed, replacing $x$ by $1/x$ we have:
\begin{align*}
\frac{1}{x^d} + \zeta x^d = f \left( \frac{1}{x} + \zeta x \right) = f \left( \frac{\zeta}{\zeta x} + \zeta x \right) = (\zeta x)^d + \frac{\zeta}{(\zeta x)^d} = \zeta^d x^d + \frac{1}{\zeta^{d-1} x^d} 
\end{align*}
which yields $\zeta^{d-1} = 1$.

In \cite{Adams2025}, Adams considered twisted Chebyshev polynomials as another family of polynomials having iterated Galois groups with bad properties. However, these polynomials are just linearly conjugate to the standard Chebyshev polynomials:

\begin{lemma}
Let $F$ be a field, $F_p$ its prime field and $d \geq 1$. Then, the twisted Chebyshev polynomial $T_{d,\zeta}$ is linearly conjugate over $F_p(\sqrt{\zeta})$ to $\pm T_d$. Moreover, if $d$ is even $T_d$ and $-T_d$ are linearly conjugate over $F_p$.
\label{lemma: Chebyshev and twisted Chebyshev}
\end{lemma}

\begin{proof}
Let $\xi_{d-1}$ be a $(d-1)$th primitive root of unity. Then $\zeta = \xi_{d-1}^s$ for some $s \in \set{0, \dots, d-2}$. Consider $a$ such that $a^2 = \zeta$. If $d$ is even, then $a = \xi_{d-1}^t$ for some $t \in \set{0, \dots, d-2}$. This is because if $a^2 = \zeta$ then we need to solve $2t \equiv s \pmod{d-1}$ which has a solution since $2$ is invertible modulo $d-1$. In particular, $a^{d-1} = 1$. On the other hand, if $d$ is odd, then  $$a^{d-1} = \left((\xi_{d-1})^{\frac{d-1}{2}}\right)^s = (-1)^s.$$

Consider $L(x) = a x$. Then 
\begin{align*}
L^{-1} \circ T_{d, \zeta} \circ L \left( x + \frac{1}{x}  \right) &= L^{-1} \circ T_{d, \zeta} \left( a x + \frac{a^2}{a x}  \right) \\
&= L^{-1} \left( a^d x^d + \frac{\zeta}{a^d x^d}  \right) \\
&= a^{d-1} x^d + \frac{1}{a^{d-1} x^d}.
\end{align*}
If $d$ is even, we obtain $L^{-1} \circ T_{d, \zeta} \circ L = T_d$. If $d$ is odd, then $L^{-1} \circ T_{d, \zeta} \circ L = (-1)^s T_d$.  

Finally, observe that when $d$ is even, the maps $-T_d$ and $T_d$ are linearly conjugate by $L(x) = -x$.
\end{proof}

\section{Polynomials not linearly conjugate over $\CC$ to $\pm T_d$} 
\label{section: Polynomials not linearly conjugate over C to T_d}

We start with the points (1) and (2) in \cref{Theorem: main result clasification}. Although it is not of our interest in this article, note that the following proof also applies if $f$ is an invertible rational function of degree $1$.

\begin{lemma}
\label{lemma: case degree 1}
Let $K$ be a number field and $f \in K[x]$ a polynomial. 
\begin{enumerate}
\item If $f$ has degree $0$, then $$\Per_{\inf}(f,K) = 0.$$ 
\item If $f$ has degree $1$, then $$\Per_{\inf}(f,K) = 1.$$  
\end{enumerate}
\end{lemma}

\begin{proof}
Let $\mathfrak{p}$ be a prime ideal in $\OK$ of good reduction for $f$. If $f$ has degree zero then $f_\mathfrak{p}(\alpha) = \alpha$ at only one point, so $$\Per_{\inf}(f,K) = 0.$$

If $f$ is a polynomial of degree $1$, then $f_\mathfrak{p}$ is invertible. Therefore, if $x_0$ is a strictly preperiodic point, there exist minimal numbers $n > m \geq 1$ such that $f_\mathfrak{p}^n(x_0) = f_\mathfrak{p}^m(x_0)$. But then, as $m$ and $n$ are minimal, this implies that $f_\mathfrak{p}^m(x_0)$ has two different preimages: $f_\mathfrak{p}^{m-1}(x_0)$ and $f_\mathfrak{p}^{n-1}(x_0)$, contradicting that $f_\mathfrak{p}$ is invertible. 

Therefore, every point is periodic and $\Per_{\inf}(f,K) = 1$.
\end{proof}

The next step is to prove point (3) in \cref{Theorem: main result clasification}, namely, that $\Per_{\inf}(f,K) = 0$ for every polynomial $f$ that is not linearly conjugate over $\CC$ to $\pm T_d$. As it was mentioned in the introduction, the proof is a simple corollary from the theory of iterated Galois groups and fixed-point proportion that we explain in the following.

Let $F$ be a field, $t$ transcendental over $F$ and fix $F(t)^{\mathrm{sep}}$ an separable closure of $F(t)$. If $f \in F(x)$ is a separable rational function with degree $d \geq 2$, we denote $f^{-n}(t)$ the set of preimages of $t$ under $f^n$. Let $$F_\infty(f,t) := F \left(\bigcup_{n \geq 0} f^{-n}(t) \right)$$ be the field extension of $F(t)$ obtained by adding all the preimages of $t$. It can be proved that $F_\infty(f,t)/F(t)$ is a Galois extension. Its Galois group, called the \textit{iterated Galois group} of $f$ over $F$ is $$G_\infty(F,f,t) := \Gal(F_\infty(f,t)/F(t)).$$

The Galois group $G_\infty(F,f,t)$ is profinite and acts transitively on $f^{-n}(t)$ for every $n \geq 1$; see \cite{Jones2014, Radi2025FPP}. Given $n \geq 1$ and $g \in G_\infty(F,f,t)$, let us denote $$X_n(g) := \# \set{z \in f^{-n}(t): g(z) = z}.$$

As $G_\infty(F,f,t)$ is profinite, we can consider $\mu$ its unique Haar probability measure. The \textit{fixed-point proportion} of $G_\infty(F,f,t)$ is defined as $$\FPP(G_\infty(F,f,t)) := \mu(\set{g \in G_\infty(F,f,t) : \# X_n(g) > 0 \text{ for all $n \geq 1$}}).$$

It has been observed since the eighties that this number has several applications to problems in arithmetic dynamics \cite{JonesManes2012, ManesThompson2019, Odoni1980, Odoni1985} and for this reason is one of the areas strongly studied currently; see \cite{Adams2025, AdamsHyde2025, Fariña2026, Tucker2013ABC, Jones2008, Jones2012} and \cite[Section 5]{BenedettoManesTucker2019}.

One of the applications of the fixed-point proportion concerns $\Per_{\inf}(f,K)$.

\begin{proposition}[{\cite[Proof of Theorem 6.5]{Juul2014} or \cite[Theorem 3.11]{BridyJones2022}}]
Let $K$ be a number field and $f \in K(x)$ a rational function with degree at least $2$. Then $$\Per_{\inf}(f,K) \leq \FPP(G_\infty(\CC,f,t))$$ where $t$ is transcendental over $\CC$.
\label{proposition: Perinf bounded by FPP}
\end{proposition}

A recent result of Fariña-Asategui and the author classifies the values of the fixed-point proportion for iterated Galois groups over algebraically closed fields:

\begin{theorem}[{\cite[Theorem 2]{FariñaRadi2026FPP}}]
\label{theorem: classification FPP geometric}
Let $F$ be a field and $f\in F[x]$ a polynomial of degree $d \geq 2$, such that either $\car(F)=0$ or $\car(F)$ does not divide the local degree of any critical point of $f$. Then, for $t$ transcendental over $F$, either
\begin{enumerate}[\normalfont(i)]
\item $f$ has a euclidean orbifold of type $(2,2,\infty)$ and $$\FPP(G_\infty(\overline{F}, f, t))=\left \{ \begin{matrix} 
1/2 & \mbox{if $d$ is odd,} \\ 
1/4 & \mbox{if $d$ is even;}
\end{matrix}\right.$$
\item or $\FPP(G_\infty(\overline{F}, f, t))=0$ otherwise.
\end{enumerate}
\end{theorem}

As it was proved in \cite[Proposition 6.6]{FariñaRadi2026FPP}, polynomials with euclidean orbifold $(2,2,\infty)$ are linearly conjugate over $\overline{F}$ to $\pm T_{d, \zeta}$ for some $d \geq 2$ and $\zeta \in \overline{F}$. By \cref{lemma: Chebyshev and twisted Chebyshev}, we have that in fact the maps $T_{d,\xi}$ are linearly conjugate over $\overline{F}$ to $\pm T_d$. 

All this together allows us to conclude point (3) in \cref{Theorem: main result clasification}:

\begin{theorem}
Let $K$ be a number field and $f \in K[x]$ a polynomial with degree $d \geq 2$ such that $f$ is not linearly conjugate over $\CC$ to $\pm T_d$ . Then $$\Per_{\inf}(f,K) = 0.$$
\end{theorem}

\begin{proof}
Combine \cref{proposition: Perinf bounded by FPP} and \cref{theorem: classification FPP geometric}.
\end{proof}

\section{Proportion of periodic points in Chebyshev polynomials}
\label{section: Proportion of periodic points in Chebyshev polynomials}

The goal of this section is to calculate $\Per(\pm T_d, F)$ for any field $F$ and any $d \geq 2$, and count its number of elements when $F$ is a finite field. The strategy will follow the same idea exposed in \cite[Example 7.2]{Juul2014}, so for this reason, we keep much of the notation used there.

Let $F$ be any field of $\car(F) \neq 2$ and $\overline{F}$ its algebraic closure. Given $d \geq 2$, define $\pi(x) := x + x^{-1}$ and $g(x) := x^d$. Then \cref{equation: definition Chebyshev polynomial} can be read as the following commutative diagram
\begin{equation}
\centering
\begin{tikzcd}
\PP^1(F) \arrow{rr}{g} \arrow{dd}{\pi} & & \PP^1(F) \arrow{dd}{\pi} \\
 & & \\
\PP^1(F) \arrow{rr}{T_d} & & \PP^1(F) \\
\end{tikzcd}
\label{diagram: definition Td}
\end{equation}

\begin{lemma}
Let $F$ be any field of $\car(F) \neq 2$. Then,
\begin{enumerate}
\item the map $\pi: \PP^1(\overline{F}) \rightarrow \PP^1(\overline{F})$ is surjective. 
\item We have $\pi(x) = \pi(y)$ if and only if $y = x$ or $y = x^{-1}$. 
\item The only critical points of $\pi$ are $\pm 1$ with local degree $2$.
\end{enumerate}
\label{lemma: properties pi}
\end{lemma}

\begin{proof}
Let $w \in \PP^1(F)$. If $w = \infty$, then $\pi(0) = \pi(\infty) = \infty$. If $w \in F$, then $\pi(x) = w$ if and only if $$x^2 -wx + 1 = 0.$$ This has a solution since we are working over $\overline{F}$. Moreover, since it is a quadratic polynomial it has at most two solutions. As $\pi(x) = \pi(x^{-1})$, this implies that $\pi(x) = \pi(y)$ if and only if $y = x$ or $y = x^{-1}$.

Finally, the derivative of $\pi$ is $$\pi'(x) = 1 - x^{-2},$$ so the critical points are $\pm 1$. As $\pi''(\pm 1) = \pm 2$, we conclude that $e_\pi(\pm 1) = 2$.
\end{proof}

We summarize well-known properties about the Chebyshev polynomials in the following proposition:

\begin{proposition}
Let $d \geq 1$ and $F$ a field of characteristic zero. 
\begin{enumerate}
\item $T_0(x) = 2$, $T_1(x) = x$ and then we have the recursion $$T_{d+1}(x) = x T_d(x) - T_{d-1}(x).$$
\item If $T_d(x) = \sum_{i = 0}^d a_{i,d} x^i$, then $a_{i,d} \in \ZZ$ and $a_{i,d} = 0$ if and only if $i \equiv d \pmod{2}$.
\item If $\alpha \in \overline{F}$ then $e_{\pm T_d}(\alpha)$ is either $1$ or $2$. In particular, if $d \geq 3$ then $\pm T_d$ is only totally ramified at $\infty$.
\end{enumerate}
\label{proposition: properties Chebyshev polynomials}
\end{proposition}

\begin{proof}
For (1), the first cases are obvious from \cref{equation: definition Chebyshev polynomial}. Then 
\begin{align*}
\left( x + \frac{1}{x} \right) T_d\left( x + \frac{1}{x} \right) = \left( x + \frac{1}{x} \right) \left( x^d + \frac{1}{x^d} \right) = T_{d+1}\left( x + \frac{1}{x} \right) + T_{d-1}\left( x + \frac{1}{x} \right).
\end{align*}

For (2), using  previous point we have the relation 
$$a_{d+1-j,d+1} = a_{d-j,d} - a_{d+1-j,d-1}.$$
Then by induction it follows that
\begin{align*}
\left \{ \begin{matrix} 
a_{d-j,d} > 0 & \mbox{if $j \equiv 0 \pmod{4}$,} \\ 
a_{d-j,d} = 0 & \mbox{if $j \equiv 1 \pmod{2}$,} \\
a_{d-j,d} < 0 & \mbox{if $j \equiv 2 \pmod{4}$.} 
\end{matrix}\right.
\end{align*} 
Finally for (3), using \cref{equation: definition Chebyshev polynomial} we obtain 
\begin{align*}
T_d'(\pi(x)) = \frac{d}{x^{d-1}} \frac{(x^2)^d-1}{x^2-1}.
\end{align*}
So, if $\pi(x) \neq \infty$ and we let $\xi_{2d}$ be a primitive $(2d)$th root of unity, then $T_d'(\pi(x)) = 0$ at the points $$x_k = (\xi_{2d})^k$$ with $k \notin d\ZZ$. By multiplicativity of the local degree and \cref{diagram: definition Td} we have $$e_{T_d}(\pi(x_k)) \cdot e_\pi(x_k) = e_\pi(g(x_k)) \cdot e_g(x_k).$$
Then $g(x_k) = \pm 1$ and $e_g(x_k) = 1$. Moreover, by \cref{lemma: properties pi} $e_\pi(x_k) = 1$ and $e_\pi(g(x_k)) = 2$, so $e_{T_d}(\pi(x_k)) = 2$. The critical points of $-T_d$ and its corresponding local degrees are the same.
\end{proof}

In principle we need to find the periodic points of $\pm T_d$. However, the following lemma shows us that we only need to worry about the periodic points of the positive case:

\begin{lemma}
\label{lemma: Per -Td and Per Td}
Given $d \geq 2$ and a field $F$, then 
$$\Per(-T_d, F) =  (-1)^{d+1}\Per(T_d, F).$$ 
\end{lemma}

\begin{proof}
If $d$ is even, by \cref{lemma: Chebyshev and twisted Chebyshev} we have that $T_d$ and $-T_d$ are linearly conjugate over $F$ by $L(x) = -x$, so $$\Per(-T_d, F) =  -\Per(T_d, F) := \set{-\alpha: \alpha \in \Per(T_d, F)}.$$ If $d$ is odd, using \cref{equation: definition Chebyshev polynomial} it is not hard to see that $$(-T_d)^2 = (T_d)^2.$$ Thus if $\alpha \in \Per(-T_d, F)$, there exists $m \geq 1$ such that $(-T_d)^m(\alpha) = \alpha$ and so 
\begin{align*}
\alpha = (-T_d)^{2m}(\alpha) = (T_d)^{2m}(\alpha)
\end{align*}
proving that $\alpha$ is also periodic for $T_d$. The same argument applies starting from $T_d$ so $\Per(-T_d, F) = \Per(T_d, F)$.
\end{proof}

The next lemma shows that the periodic points of $T_d$ are obtained from the periodic points of $g$. 

\begin{lemma}
Let $d \geq 1$ and $F$ a field. Let $\alpha \in F$ and $\beta \in \overline{F}$ such that $\pi(\beta) = \alpha$. Then $\alpha \in \Per(T_d, F)$ if and only if $\beta \in \Per(g, \overline{F})$.
\label{lemma: Per Td bijection Per g}
\end{lemma}

\begin{proof}
If $\alpha \in \Per(T_d, F)$, there exists $m \geq 1$ such that $(T_d)^m(\alpha) = \alpha$. Then by \cref{diagram: definition Td}, 
\begin{align*}
\pi(\beta) = \alpha = (T_d)^m(\alpha) = (T_d)^m(\pi(\beta)) = \pi(g^m(\beta)).
\end{align*}
By \cref{lemma: properties pi} this means that $g^m(\beta) \in \set{\beta, \beta^{-1}}$. In the first case $\beta$ is periodic for $g$. If $g^m(\beta) = \beta^{-1}$ then 
\begin{align*}
g^{2m}(\beta) = g^m(\beta^{-1}) = (g^m(\beta))^{-1} = \beta.
\end{align*}

Conversely, if $\beta \in \Per(g,\overline{F})$, there exists $m \geq 1$ such that $g^m(\beta) = \beta$. Again by \cref{diagram: definition Td}
\begin{equation*}
\alpha = \pi(\beta) = \pi(g^m(\beta)) = (T_d)^m(\pi(\beta)) = (T_d)^m(\alpha).  \qedhere
\end{equation*}
\end{proof}

From now and on, assume that $F$ is a finite field with $\car(F) \neq 2$ and let $F_2$ be the extension of degree $2$ of $F$. By \cref{lemma: properties pi} $\pi(F_2^\times)$ contains $F^\times$ and it is 2-to-1 except at $\pm 1$. Consider the subgroup 
\begin{align*}
H_2 := \set{\beta \in F_2^\times: \beta^{\abs{F}+1 } = 1}
\end{align*}
and $U_2 := H_2 \setminus \set{\pm 1}$.

\begin{lemma}
Let $F$ be a finite field with $\car(F) \neq 2$. Then, the set $$\set{\beta \in F_2^\times: \pi(\beta) \in F} = F^\times \bigcup U_2.$$ Moreover the union is disjoint.
\label{lemma: Fp and U2}
\end{lemma}

\begin{proof}
If $\beta \in F^\times$ clearly $\pi(\beta) \in F$. Now $\beta \in F_2^\times \setminus F^\times$ if and only if the action of the Frobenius map is non-trivial. As the Frobenius map must send $\beta$ to the other solution of $\pi(x) = \pi(\beta)$, by \cref{lemma: properties pi} this implies that 
\begin{align}
\beta \in F_2^\times \setminus F^\times \Longleftrightarrow \beta^{\abs{F}} = \beta^{-1} \text{ and } \beta \neq \pm 1.
\label{equation: beta in Fp2 - Fp}
\end{align}
Similarly, we have $\pi(\beta) \in F$ if and only if the action of the Frobenius map is trivial, namely, 
\begin{align}
\pi(\beta) \in F \Longleftrightarrow \beta + \beta^{-1} = \beta^{\abs{F}} + \beta^{-\abs{F}}.
\label{equation: pi(beta) in Fp}
\end{align}
Combining \cref{equation: beta in Fp2 - Fp} and \cref{equation: pi(beta) in Fp} we obtain that 
\begin{align*}
\beta \in F_2^\times \setminus F^\times \text{ and } \pi(\beta) \in F \Longleftrightarrow \beta^{\abs{F}+1} = 1 \text{ and } \beta \neq \pm 1.
\end{align*}
To prove that the union is disjoint observe that if $\beta \in F^\times \cap U_2$ then $\beta^{\abs{F}} = \beta$ and $\beta^{\abs{F}+1} = 1$, which implies that $\beta^2 = 1$ and consequently $\beta = \pm 1$. However, by definition $\pm 1 \notin U_2$.
\end{proof}

Combining \cref{lemma: Per Td bijection Per g} and \cref{lemma: Fp and U2}, we can give a description of the periodic points of $T_d$ in $F$:

\begin{lemma}
Let $d \geq 2$, $F$ a finite field with $\car(F) \neq 2$, $F_2$ the extension of degree $2$ of $F$ and let $\abs{ \cdot}$ denote the order of an element in the group $F_2^\times$. Then 
\begin{align*}
\Per(T_d, F) = \bigcup_{\gamma \in \set{-1, 1}} \set{\pi(\beta) \in F: \gcd(\abs{\beta},d) = 1 \text{ and } \abs{\beta} \mid (\abs{F} + \gamma)} 
\end{align*}
and the sets involved in the union intersects at $1$ if $d$ is even or at $\pm 1$ if $d$ is odd.
\label{lemma: Per Td}
\end{lemma}

\begin{proof}
We have the following chain of equivalences
\begin{align*}
\beta \in \Per(g, F_2^\times) & \Longleftrightarrow g^m(\beta) = \beta^{d^m} = \beta \text{ for some $m \geq 1$,} \\
& \Longleftrightarrow \abs{\beta} \mid (d^m-1) \text{ for some $m \geq 1$,} \\
& \Longleftrightarrow d^m \equiv 1 \pmod{\abs{\beta}} \text{ for some $m \geq 1$,} \\
& \Longleftrightarrow \gcd(\abs{\beta},d) = 1.
\end{align*}

Among the elements in the set $\Per(g, F_2^\times)$, we are only interested in those where $\pi(\beta) \in F$. By \cref{lemma: Fp and U2}, we have two disjoint cases. If $\beta \in F^\times$, then $\abs{\beta}$ must divide $\abs{F}-1$. If $\beta \in U_2$ then $\abs{\beta}$ must divide $\abs{F}+1$.
\end{proof}

Our next goal in this section is to calculate the number of periodic points of $\pm T_d$ in $F$. 

\begin{Definition}
Given $a,b \geq 1$, define 
\begin{align*}
r(a,b) := \text{ largest divisor of $a$ coprime to $b$}.
\end{align*}
\end{Definition}

\begin{lemma}
\label{lemma: number of elements with order coprime to a number}
Let $G$ be a cyclic finite group of order $n$. Let $m \geq 1$ such that $m \mid n$ and let $b \geq 1$. Define $$H := \set{g \in G: g^m = 1}.$$
Then, 
\begin{align*}
\# \set{h \in H: \gcd(\abs{h},b) = 1} = r(m,b).
\end{align*}
\end{lemma}

\begin{proof}
As $G$ is cyclic, then $H$ is a cyclic subgroup of $G$. Moreover, since $m \mid n$, then $\abs{H} = m$. Finally
\begin{align*}
\# \set{h \in H: \gcd(\abs{h},b) = 1 } &= \sum_{u \mid r(m,b)} \#\set{h \in H: \abs{h} = u} \\
&= \sum_{u \mid r(m,b)} \varphi(u) = r(m,b). \qedhere
\end{align*}
\end{proof}

So the number of periodic points for $\pm T_d$ in a finite field $F$ is the following:

\begin{lemma}
\label{lemma: number of Per for Td}
Let $d \geq 2$ and $F$ a finite field with $\car(F) \neq 2$. Then 
$$\# \Per(\pm T_d, F) = \frac{r(\abs{F}-1,d) + r(\abs{F}+1,d) }{2}.$$
\end{lemma}

\begin{proof}
By \cref{lemma: Per -Td and Per Td} we only need to worry about the periodic points of $T_d$. As before, let $F_2$ be the extension of $F$ of degree $2$. 

Consider $$H_1 = \set{\beta \in F_2^\times: \beta^{\abs{F}-1} = 1}.$$ As $F_2^\times$ is a cyclic subgroup, by \cref{lemma: number of elements with order coprime to a number} we have
\begin{align*}
\#\set{\beta \in F_2^\times: \gcd( \abs{\beta},d) = 1 \text{ and } \abs{\beta} \mid (\abs{F} -1)} = r(\abs{F}-1, d).
\end{align*}
The same argument yields for $H_2$ so 
\begin{align*}
\#\set{\beta \in F_2^\times: \gcd( \abs{\beta},d) = 1 \text{ and } \abs{\beta} \mid (\abs{F} +1)} = r(\abs{F}+1, d).
\end{align*}

To calculate $\#\Per(T_d, F)$, observe that $\alpha = \pi(\beta) = \pi(\beta^{-1})$, so every element with $\beta \neq \pm 1$ is counted twice. Then $\beta = 1$ gives $\gcd(\abs{\beta},d) = 1$ and $\beta = -1$ gives $\gcd(\abs{\beta},d) = 1$ if and only if $d$ is odd. If we define $\chi(d) = 1$ if $d$ is even and $\chi(d) = 2$ if $d$ is odd, then, $\chi(d)$ indicates if we need to isolate $1$ or $2$ elements depending on the parity of $d$, to avoid an overcount. Therefore,
\begin{equation*}
\# \Per(T_d, F) = \frac{r(\abs{F}-1,d) -\chi(d)}{2} + \frac{r(\abs{F}+1,d) -\chi(d)}{2} + \chi(d). \qedhere
\end{equation*}
\end{proof}

\begin{remark}
\label{remark: Perinf Td}
If $K$ is a number field and $d \geq 2$, then by \cref{lemma: number of Per for Td} we obtain
\begin{align*}
\Per_{\inf}(\pm T_d,K) = \frac{1}{2} \liminf_{N(\mathfrak{p}) \rightarrow +\infty} \frac{r(N(\mathfrak{p})-1,d) + r(N(\mathfrak{p})+1,d)}{N(\mathfrak{p})}
\end{align*}
\end{remark}

\section{Polynomials linearly conjugate over $\CC$ to $\pm T_d$} 
\label{section: Polynomials linearly conjugate over C to T_d}

In this section we prove point (4) in \cref{Theorem: main result clasification}:

\begin{lemma}
\label{lemma: Perinf when conjugating in K}
Let $K$ be a number field and $f,g \in K[x]$ such that they are linearly conjugate over $K$. Then
\begin{align*}
\Per_{\inf}(f,K) = \Per_{\inf}(g,K).
\end{align*}
In particular if $f$ is linearly conjugate over $K$ to $\pm T_d$, then 
$$\Per_{\inf}(f,K) = \Per_{\inf}(T_d,K).$$
\end{lemma}

\begin{proof}
Let $L$ be a linear conjugation over $K$ between $f$ and $g$. Then, if $\mathfrak{p}$ is a prime ideal in $\OK$ of good reduction for $f$, $g$ and $L$, then $L$ induces a bijection between $\Per(f_\mathfrak{p},\FF_\mathfrak{p})$ and $\Per(g_\mathfrak{p},\FF_\mathfrak{p})$. As $\Per_{\inf}(-T_d,K) = (-1)^{d+1}\Per_{\inf}(T_d,K)$ by \cref{lemma: Per -Td and Per Td}, then the second part follows.
\end{proof}

Let $f$ be a polynomial with coefficients in a number field $K$ and degree $d$. First of all, observe that we may assume that the $(d-1)$th coefficient of $f$ is zero. Indeed, conjugate $f$ by a map of the form $L(x) = x+b$. If $f(x) = \sum_{i = 0}^d b_i x^i$, then 
\begin{align*}
L^{-1} \circ f \circ L(x) = b_dx^d + (d \cdot b \cdot b_d + b_{d-1})x^{d-1} + \cdots.
\end{align*}
So, we can take $$b = -\frac{b_{d-1}}{d \cdot b_d}.$$

Suppose $f$ is linearly conjugate over $\CC$ to $\pm T_d$ by a linear map $L$. If $d = 2$ and $L$ is not a polynomial, then $L$ must exchange $0$ and $\infty$ as those are the totally ramified critical points of $\pm T_2$. Thus $L(x) = a/x$ and 
\begin{align*}
L^{-1} \circ (\pm T_2) \circ L(x) = \pm \frac{a x^2}{a^2-2x^2},
\end{align*}
which is not a polynomial. If $d \geq 3$, by \cref{proposition: properties Chebyshev polynomials} we have $\pm T_d$ is only totally ramified at $\infty$ and therefore $L$ is a polynomial by \cref{lemma: L is a polynomial}. Write $L(x) = ax+b$ and $\pm T_d = \sum_{i = 0}^d a_i x^i$. Recall that by \cref{proposition: properties Chebyshev polynomials} we have $a_{d-1} = 0$. Then
\begin{align*}
f(x) = L^{-1} \circ (\pm T_d) \circ L(x) = a^{d-1} \cdot a_d \cdot x^d + a^{d-2} \cdot a_d \cdot d \cdot b \cdot x^{d-1} + \cdots 
\end{align*}
which implies that $b = 0$ as we are assuming that the $(d-1)$th coefficient of $f$ is zero. Then 
\begin{align*}
f(x) = L^{-1} \circ (\pm T_d) \circ L(x) = \sum_{i = 0}^d a^{i-1} \cdot a_i \cdot x^i.
\end{align*}

If $d$ is even, then $a \cdot a_2 \in K$ and as $a_2 \neq 0$ by \cref{proposition: properties Chebyshev polynomials}, then $a \in K$. By \cref{remark: Perinf Td} and \cref{lemma: Perinf when conjugating in K} we conclude that 
\begin{align}
\Per_{\inf}(f,K) = \frac{1}{2} \liminf_{N(\mathfrak{p}) \rightarrow +\infty} \frac{r(N(\mathfrak{p})-1,d) + r(N(\mathfrak{p})+1,d)}{N(\mathfrak{p})}.
\label{equation: Perinf f with f even degree and conj to Td}
\end{align}

If $d$ is odd, then $a_2 = 0$ but $a_3 \neq 0$, so we conclude that $a^2 \cdot a_3 \in K$ and thus $a^2 \in K$. If $a \in K$ then $\Per_{\inf}(f,K)$ is calculated as in \cref{equation: Perinf f with f even degree and conj to Td}, so, let us assume that $a \notin K$.

\begin{proposition}
\label{proposition: Perinf f with a not in K}
Let $K$ be a number field, an odd number $d \geq 2$ and consider $f = L^{-1} \circ (\pm T_d) \circ L \in K[x]$ with $L(x) = ax$ and $a \notin K$. Then,  
\begin{align*}
\Per_{\inf}(f,K) = \frac{1}{2} \liminf_{N(\mathfrak{p}) \rightarrow +\infty} \frac{r(N(\mathfrak{p})-1,d) + r(N(\mathfrak{p})+1,d)}{N(\mathfrak{p})}.
\end{align*}
\end{proposition}

\begin{proof}
Call $K' = K(a)$. By hypothesis $K'/K$ is an extension of degree $2$. Let $\mathfrak{p}$ be a prime ideal such that $\car(\FF_\mathfrak{p}) \neq 2$ and $\mathfrak{p}$ does not ramify at $K'$. Notice that we are excluding only finitely many prime ideals. Let $\mathfrak{p}'$ be a prime ideal in $\mathcal{O}_{K'}$ such that $\mathfrak{p}' \mid \mathfrak{p}$.

As $f$ and $\pm T_d$ are linearly conjugate over $K'$, then $L$ induces a bijection
\begin{align}
\label{equation: bijection Per f and Per Td d odd a not in K}
\Per(f_{\mathfrak{p}'}, \FF_{\mathfrak{p}'}) &\rightarrow  \Per((\pm T_d)_{\mathfrak{p}'}, \FF_{\mathfrak{p}'}) \\
\notag \alpha &\mapsto a \alpha.
\end{align}
Moreover, since $d$ is odd by \cref{lemma: Per -Td and Per Td} we have $\Per((\pm T_d)_{\mathfrak{p}'}, \FF_{\mathfrak{p}'}) = \Per((T_d)_{\mathfrak{p}'}, \FF_{\mathfrak{p}'})$.

We have two cases for $\mathfrak{p}$. If $\mathfrak{p}$ is split, then $\FF_{\mathfrak{p}'} = \FF_\mathfrak{p}$ and therefore by \cref{lemma: number of Per for Td},
\begin{align}
\# \Per(f_\mathfrak{p}, \FF_\mathfrak{p}) = \frac{r(N(\mathfrak{p})-1,d) + r(N(\mathfrak{p})+1,d)}{2}.
\label{equation: Per f in split case}
\end{align}

If $\mathfrak{p}$ is inert, then $\FF_{\mathfrak{p}'}/\FF_\mathfrak{p}$ is an extension of degree $2$ and consequently the Frobenius map sends $a$ to $-a$ in $\FF_\mathfrak{p'}$.

Recall the trace is (in this case) the surjective additive group homomorphism
\begin{align*}
\tr: \FF_{\mathfrak{p}'} &\rightarrow \FF_\mathfrak{p}, \\
y &\mapsto y + \Frob_\mathfrak{p}(y).
\end{align*}
As it is surjective we have $$\# \ker(\tr) = N(\mathfrak{p}).$$

We claim that
\begin{align*}
a \FF_\mathfrak{p} = \ker(\tr).
\end{align*}

Indeed, if $y\in a \FF_\mathfrak{p}$ then $y = ax$ for some $x \in \FF_\mathfrak{p}$. Therefore 
\begin{align*}
\tr(y) = ax -ax^{N(\mathfrak{p})} = 0
\end{align*}
as $x \in \FF_\mathfrak{p}$. This proves that $a\FF_\mathfrak{p} \subseteq \ker(\tr)$, but since they have the same cardinality, then $a\FF_\mathfrak{p} = \ker(\tr)$.

Let $F_2$ be the extension of degree $2$ of $\FF_{\mathfrak{p}'}$ and recall the map $\pi: \PP^1(F_2) \rightarrow \PP^1(F_2)$. Then, if $\beta \in F_2^\times$,
\begin{align*}
\tr(\pi(\beta)) = \beta + \beta^{-1} + \beta^{N(\mathfrak{p})} + \beta^{-N(\mathfrak{p})} = \frac{(\beta^{N(\mathfrak{p})-1} + 1)(\beta^{N(\mathfrak{p})+1} + 1)}{\beta^{N(\mathfrak{p})}}.
\end{align*}

We are interested in calculating $\# \Per(f_\mathfrak{p}, \FF_\mathfrak{p})$. Using the bijection in \cref{equation: bijection Per f and Per Td d odd a not in K}, we obtain that 
\begin{align*}
\# \Per(f_\mathfrak{p}, \FF_\mathfrak{p}) = \#(a \FF_\mathfrak{p} \cap \Per((T_d)_{\mathfrak{p}'}, \FF_{\mathfrak{p}'})).
\end{align*}

Using \cref{lemma: Per Td}, we reduce the previous problem to calculate the number of elements of
\begin{align*}
V_\gamma := \set{\pi(\beta) \in \FF_{\mathfrak{p}'}: \tr(\pi(\beta)) = 0, \,\, \gcd(\abs{\beta},d) = 1 \text{ and } \abs{\beta} \mid (N(\mathfrak{p})^2+\gamma)} 
\end{align*}
for $\gamma \in \set{-1,1}$.

Start with $\gamma = 1$. We claim that 
$$\set{\beta \in F_2^\times: \tr(\pi(\beta)) = 0 \text{ and } \abs{\beta} \mid (N(\mathfrak{p})^2+1)} = \emptyset,$$ which in particular implies that $V_1 = \emptyset$. Indeed, if $\tr(\pi(\beta)) = 0$ then $\beta^{N(\mathfrak{p}) \pm 1} = -1$ for either $+$ or $-$ in the exponent. Therefore $$\abs{\beta} \mid 2(N(\mathfrak{p})\pm 1).$$
Notice that as $N(\mathfrak{p})$ is odd, the previous formula implies that $$\abs{\beta} \mid (N(\mathfrak{p})^2 - 1).$$
On the other hand if $\abs{\beta} \mid (N(\mathfrak{p})^2 + 1)$, subtracting we conclude that $\abs{\beta} \mid 2$ and therefore $\beta = \pm 1$. However $\tr(\pi(\pm 1)) \neq 0$.

If $\gamma = -1$ consider 
$$H_{-1} := \set{\beta \in F_2^\times: \tr(\pi(\beta)) = 0 \text{ and } \abs{\beta} \mid (N(\mathfrak{p})^2-1)}.$$
By the calculation done before if $\tr(\pi(\beta)) = 0$ then $\abs{\beta} \mid (N(\mathfrak{p})^2-1)$, so
$$H_{-1} = \set{\beta \in F_2^\times: \tr(\pi(\beta)) = 0}.$$
Let $$G_c := \set{\beta \in F_2^\times: \beta^c = 1}.$$ 

Then
$$H_{-1} = \bigcup_{\gamma \in \set{-1,1}} (G_{2(N(\mathfrak{p})+\gamma)} \setminus G_{N(\mathfrak{p})+\gamma})$$ as $\tr(\pi(\beta)) = 0$ if and only if $\beta^{N(\mathfrak{p}) + \gamma} = -1$ for either $\gamma = 1$ or $\gamma = -1$.

Moreover $$\beta \in \bigcap_{\gamma \in \set{-1,1}} (G_{2(N(\mathfrak{p})+\gamma)} \setminus G_{N(\mathfrak{p})+\gamma})$$ if and only if $\abs{\beta} = 4$. 

Notice that as $\pm 1 \notin H_{-1}$, if $\pi(\beta) \in V_{-1}$ then $\pi(\beta)$ has two different preimages in $H_{-1}$. Using this observation and \cref{lemma: number of elements with order coprime to a number} we obtain
\begin{align*}
\# V_{-1} = \frac{1}{2} \left(\sum_{\gamma \in \set{-1,1}} (r(2(N(\mathfrak{p}) + \gamma),d) - r(N(\mathfrak{p}) + \gamma,d)) \right) - \frac{\#\set{\beta \in F_2^\times: \abs{\beta} = 4}}{2}.
\end{align*}
As $d$ is odd then $r(2(N(\mathfrak{p}) + \gamma),d) = 2 r(N(\mathfrak{p}) + \gamma,d)$ and therefore
\begin{align*}
 \# V_{-1} = \frac{1}{2} \left( \sum_{\gamma \in \set{-1,1}} r(N(\mathfrak{p}) + \gamma,d) \right) - \frac{\#\set{\beta \in F_2^\times: \abs{\beta} = 4}}{2}.
\end{align*}
As $V_1$ was empty, then $\# \Per(f_\mathfrak{p}, \FF_\mathfrak{p}) = \# V_{-1}$ allowing us to conclude
\begin{align}
\# \Per(f_\mathfrak{p}, \FF_\mathfrak{p}) = \frac{1}{2} \left( \sum_{\gamma \in \set{-1,1}} r(N(\mathfrak{p}) + \gamma,d) \right) - \frac{\#\set{\beta \in F_2^\times: \abs{\beta} = 4}}{2}.
\label{equation: Per f in inert case}
\end{align}

Dividing \cref{equation: Per f in split case} and \cref{equation: Per f in inert case} by $N(\mathfrak{p})$ and taking limit, we obtain
$$\Per_{\inf}(f,K) = \frac{1}{2} \liminf_{N(\mathfrak{p}) \rightarrow +\infty} \frac{r(N(\mathfrak{p})-1,d) + r(N(\mathfrak{p})+1,d)}{N(\mathfrak{p})}$$ as desired.
\end{proof}

\section{Analysis of $\Per_{\inf}(T_d,K)$}
\label{section: analysis of Perinf(Td)}

In this last section, we prove the points from (5) to (10) in \cref{Theorem: main result clasification} completing the classification.

Fix $K$ a number field and recall $m_K$ is the greatest positive number such that $K$ contains a $m_K$th primitive root of unity. By \cref{remark: Perinf Td} and \cref{proposition: Perinf f with a not in K} we have that if $f$ is linearly conjugate to $\pm T_d$ over $\CC$ then 
$$\Per_{\inf}(f,K) = \frac{1}{2} \liminf_{N(\mathfrak{p}) \rightarrow +\infty} \frac{r(N(\mathfrak{p})-1,d) + r(N(\mathfrak{p})+1,d)}{N(\mathfrak{p})},$$ so we just need to analyze the possible values of this limit in the different cases considered in \cref{Theorem: main result clasification}. For this task, we will use \cref{proposition: cebotarev in my case} many times.

Let $\mathfrak{p}$ be a prime of $\OK$. Since we are taking a limit, we may assume the norm of $\mathfrak{p}$ is coprime $m_K$. Notice that in particular $N(\mathfrak{p})$ is coprime to $2$ as $2 \mid m_K$ since $-1 \in \QQ$. Under this assumption, the polynomial $$x^{m_K}-1$$ is separable in $\FF_\mathfrak{p}$ and therefore all the $m_K$th root of unity are different in $\FF_\mathfrak{p}$. As $\xi_{m_K} \in \OK$ this implies that
\begin{equation}
m_K \mid (N(\mathfrak{p})-1).
\label{equation: mK divides N(p)-1}
\end{equation}

\bigskip

\noindent 6.1. \underline{$d$ has at least two prime factors and $\rad(d) \nmid \rad(m_K)$} 

Let $p_1$ and $p_2$ be two different prime factors of $d$ such that $p_2 \nmid m_K$. Let $m_1 \geq v_{p_1}(m_K)$ and $m_2 \geq 1$. Consider the system of congruences
\begin{align*}
\left \{ \begin{matrix} 
x &\equiv & 1 \pmod{m_K} \\
x &\equiv & 1 \pmod{p_1^{m_1}} \\
x &\equiv & -1 \pmod{p_2^{m_2}} \\
\end{matrix}\right..
\end{align*} 
As $p_2 \nmid m_K$, the system has a solution. By \cref{proposition: cebotarev in my case} there are infinitely many primes $\mathfrak{p}$ in $\OK$ such that 
$$p_1^{m_1} \mid (N(\mathfrak{p})-1) \text{ and } p_2^{m_2} \mid (N(\mathfrak{p})+1).$$

As $p_1$ and $p_2$ are factors of $d$ then 
$$r(N(\mathfrak{p})-1,d) \leq \frac{N(\mathfrak{p})-1}{p_1^{m_1}} \text{ and } r(N(\mathfrak{p})+1,d) \leq \frac{N(\mathfrak{p})+1}{p_2^{m_2}}.$$ Therefore 
\begin{align*}
\Per_{\inf}(f,K) \leq \frac{1}{2} \left( \frac{1}{p_1^{m_1}} + \frac{1}{p_2^{m_2}} \right).
\end{align*}
As $m_1$ and $m_2$ are arbitrary, then 
\begin{align*}
\Per_{\inf}(f,K) = 0.
\end{align*}

\bigskip

\noindent 6.2. \underline{$d$ is odd and $\rad(d) \mid \rad(m_K)$}

We claim that $$r(N(\mathfrak{p})+1,d) = N(\mathfrak{p})+1.$$ Indeed, suppose that $r(N(\mathfrak{p})+1,d) < N(\mathfrak{p})+1$. Then, there exists $p$ prime such that 
\begin{align*}
p \mid (N(\mathfrak{p})+1) \text{ and } p \mid d.
\end{align*}
As $d$ is odd, then $p$ is odd. 

On the other hand, if $p$ divides $d$, then $p$ also divides $m_K$ by hypothesis so by \cref{equation: mK divides N(p)-1} $$p \mid (N(\mathfrak{p})-1).$$ Therefore $p \mid 2$, leading to a contradiction since $p$ is odd. 

Let $p$ be any prime divisor of $d$. For any $m \geq 1$, consider the system of congruences
\begin{align*}
\left \{ \begin{matrix} 
x &\equiv & 1 \pmod{m_K} \\
x &\equiv & 1 \pmod{p^m} \\
\end{matrix}\right..
\end{align*}
As the system has a solution, by \cref{proposition: cebotarev in my case} there are infinitely many primes $\mathfrak{p}$ in $\OK$ such that 
$$p^m \mid (N(\mathfrak{p})-1)$$ and consequently $$r(N(\mathfrak{p})-1,d) \leq \frac{N(\mathfrak{p})-1}{p^m}.$$
for any $m \geq 1$.
In conclusion, 
\begin{align*}
\Per_{\inf}(f,K) = \frac{1}{2}.
\end{align*}

\bigskip

\noindent 6.3. \underline{$d$ is even, $\rad(d) \mid \rad(m_K)$ and $v_2(m_K) \geq 2$}

We claim that $$r(N(\mathfrak{p})+1,d) = \frac{N(\mathfrak{p})+1}{2}.$$ If not there exists $p$ prime such that 
\begin{align*}
p \mid \frac{N(\mathfrak{p})+1}{2} \text{ and } p \mid d.
\end{align*}
Since $p$ divides $d$, as in the previous case we obtain that $$p \mid (N(\mathfrak{p})-1)$$ and so $p = 2$. But now if $p = 2$ then $$4 \mid (N(\mathfrak{p})+1)$$ and since $4 \mid m_K$ by hypothesis, then $$4 \mid (N(\mathfrak{p})-1),$$ leading to a contradiction.

Now taking $p$ a prime divisor of $d$, considering the same system of congruences as in the previous case and arguing the same way we obtain
\begin{align*}
\Per_{\inf}(f,K) = \frac{1}{4}.
\end{align*}

\bigskip

\noindent 6.4 \underline{$d$ is even, has at least two prime factors, $\rad(d) \mid \rad(m_K)$ and $v_2(m_K) = 1$}

Let $p$ be an odd prime dividing $d$. Let $m_1 \geq v_p(m_K)$ and $m_2 \geq v_2(m_K)$. Consider the system of congruences
\begin{align*}
\left \{ \begin{matrix} 
x &\equiv & 1 \pmod{m_K} \\
x &\equiv & 1 \pmod{p^{m_1}} \\
x &\equiv & -1 \pmod{2^{m_2}} \\
\end{matrix}\right..
\end{align*} 
As $\gcd(m_K, 2^{m_2}) = 2$ the system has a solution and so by \cref{proposition: cebotarev in my case} there are infinitely many primes $\mathfrak{p}$ in $\OK$ such that 
$$p^{m_1} \mid (N(\mathfrak{p})-1) \text{ and } 2^{m_2} \mid (N(\mathfrak{p})+1).$$ Proceeding as in the first case we obtain 
\begin{align*}
\Per_{\inf}(f,K) = 0.
\end{align*}

\bigskip

\noindent 6.5 \underline{$d$ is an odd prime power}

Let $p$ be the prime dividing $d$. Observe that either $r(N(\mathfrak{p})-1,d) = N(\mathfrak{p})-1$ or $r(N(\mathfrak{p})+1,d) = N(\mathfrak{p})+1$. Indeed, if not
\begin{align*}
p \mid (N(\mathfrak{p})-1) \text{ and } p \mid (N(\mathfrak{p})+1),
\end{align*}
which contradicts that $p$ is odd. Then, arguing as in the second case we obtain $$\Per_{\inf}(T_d,K) = \frac{1}{2}.$$ 

\bigskip

\noindent 6.6 \underline{$d$ is an even prime power}

As before, we have either $$r(N(\mathfrak{p})-1,d) = \frac{N(\mathfrak{p})-1}{2} \text{ or } r(N(\mathfrak{p})+1,d) = \frac{N(\mathfrak{p})+1}{2}.$$ If not, then $4$ divides $N(\mathfrak{p})-1$ and $N(\mathfrak{p})+1$ which leads to a contradiction. 

Finally, taking $m \geq v_2(m_K)$ and consider the system of congruences
\begin{align*}
\left \{ \begin{matrix} 
x &\equiv & 1 \pmod{m_K} \\
x &\equiv & 1 \pmod{2^m} \\
\end{matrix}\right..
\end{align*} 
we conclude that
$$\Per_{\inf}(T_d,K) = \frac{1}{4}.$$

\subsection*{Acknowledgements} 

The author would like to thank Rostislav Grigorchuk and Thomas Tucker for their support and feedback, as well as Jorge Fariña-Asategui for useful suggestions.


\bibliographystyle{unsrt}

\end{document}